\documentclass[11pt, reqno]{amsart}

\usepackage{orcidlink} 
\usepackage{amssymb, amsthm, amsmath}
\usepackage{comment}
\usepackage{graphics}
\usepackage{hyperref}
\usepackage{xcolor}
\usepackage[all]{xy}
\usepackage[mathscr]{eucal}
\usepackage[shortlabels]{enumitem}
\usepackage{thmtools}
\usepackage{amsrefs}
\usepackage{setspace}
\usepackage{mathtools}
\usepackage{stmaryrd}
\usepackage{upgreek}
\usepackage{tikz}
\usepackage{mathpazo}
\usepackage{bbm}

\tikzset{
v/.style={draw, fill, circle, minimum size=1.5mm, inner sep=0},
b/.style={draw , circle, minimum size=2.5mm, inner sep=.5mm},
e/.style={very thick},
vs/.style={draw, fill, circle, minimum size=1mm, inner sep=0},
bs/.style={draw, fill, circle, minimum size=1.5mm, inner sep=0mm},
es/.style={thick}
}
\newlength{\nodeheight}
\newlength{\nodewidth}
\usetikzlibrary{arrows,matrix,positioning}
\usetikzlibrary {shapes.geometric} 

\numberwithin{equation}{section}

\declaretheorem[name=Theorem, sibling=equation]{theorem}
\declaretheorem[name=Lemma, sibling=equation]{lemma}
\declaretheorem[name=Proposition, sibling=equation]{proposition}
\declaretheorem[name=Corollary, sibling=equation]{corollary}
\declaretheorem[name=Definition, style=definition, sibling=equation]{definition}

\DeclareMathOperator{\Tor}{Tor}

\newcommand{\M}{{\mathcal{M}}}
\newcommand{\R}{{\mathscr{R}}}

\newcommand{\RBr}{\R{\rm Br}}
\newcommand{\Br}{{\rm Br}}
\newcommand{\fS}{\mathfrak{S}}
\renewcommand{\t}{\mathbbm{1}}
\newcommand{\calA}{\mathcal{A}}
\newcommand{\calB}{\mathcal{B}}
\newcommand{\calM}{\mathcal{M}}
\newcommand{\bY}{\mathbb{Y}}

\title{Homology of Rook-Brauer Algebras and Motzkin Algebras}

\author{Khoa Ta, \orcidlink{0009-0006-8646-0447}}
\address{CSU Fullerton, 800 N. State College Blvd. Fullerton CA 92831}
\email{khoa.ta@fullerton.edu}

\begin{document}

\begin{abstract}
Using the technique of inductive resolution introduced in \cite{bhp}, we prove that the homology of Rook-Brauer Algebra, interpreted as appropriate Tor-group, is isomorphic to that of symmetric group for all degrees under the assumption that $\epsilon$ in $R$ is invertible; furthermore, we also prove the homology of the Motzkin algebras vanishes in positive degrees under the same assumption.  These results thereby establish homological stability of both algebras. 
\end{abstract}

\maketitle
\section{Introduction}
The Rook-Brauer Algebras and Motzkin Algebras are part of the family of diagram algebras that includes the Partition Algebra, Rook Algebra, Temperley-Lieb Algebra, Brauer Algebra, and many more.  In recent years, homology of these algebras has been studied extensively by various authors (see \cite{bhp} for partition algebras, \cite{bh} for Temperley-Lieb Algebras, \cite{bhp1} for Brauer Algebras, \cite{b} for Rook and Rook-Brauer Algebras with restriction on parameters, see \cite{fg} for Tanabe algebras, uniform block permutation algebras and totally propagating partition algebras). \\
\indent The results of most diagram algebras mentioned above are parameters-dependent i.e certain conditions, most notably the invertibility of the parameter, must be imposed to ensure the validity of these results.  The Rook-Brauer algebra has two defining parameters, $\epsilon$ and $\delta$ (which will be recalled later), and in \cite{b}, Guy Boyde attempted to use the concept of idempotents to demonstrate that the homology of Rook-Brauer algebras is isomorphic to that of symmetric groups in all degrees, provided that $\epsilon$ is invertible. However, as noted by Andrew Fisher and Daniel Graves in their paper (see Remark 7.1.1 \cite{fg1}), the argument presented was found to be incorrect. Fisher and Graves, by additionally assuming the invertibility of $\delta$, were able to refine and correct the approach, ultimately proving that the homology of the Rook-Brauer algebras is indeed isomorphic to that of the symmetric groups in all degrees (see Theorem 7.1.5 \cite{fg1}).  Furthermore, under the same assumption on $\delta$ and $\epsilon$, they also proved that the homology of the Motzkin algebras vanishes in positive degrees (see Theorem 7.1.6 \cite{fg1}). \\
\indent In this paper, we employ the technique of {\it inductive resolution} pioneered by Boyd, Hepworth and Patzt in \cite{bhp} to prove the same results as those of Fisher and Graves for Rook-Brauer and Motzkin algebras while \emph{only requiring the invertibility of $\epsilon$}.  More specifically, under the assumption that $\epsilon \in R$ is invertible and for any $\delta \in R$ where $R$ is a unital commutative ring, we prove that the homology of the Rook-Brauer algebra is isomorphic to that of the symmetric group for all degrees and the homology of the Motzkin Algebras vanishes in positive degrees. 

\section{Main Result}
\indent We now introduce the background necessary to discuss our main results.  For a unital commutative ring $R$, the {\it Rook-Brauer Algebras} $\RBr_n = \RBr_n(R, \delta,\epsilon)$ with parameters $\delta, \epsilon \in R$ is a free $R$-module with basis consisting of partitions of the unions of the sets $[-n] \cup [n]$ whose blocks (components) have size $\le 2$ where $[-n] = \{-n,\ldots,-1\}$ and $[n] = \{1,\ldots,n\}$. Each basis element of $\RBr_n$ may be visualized by a diagram $\alpha$ with two vertical columns of $n$ nodes each, with the nodes on the left representing $-n$ through $-1$ and the nodes on the right representing $1$ through $n$, with paths connecting certain nodes.  Since each block has size at most 2, each basis element is represented by a diagram in which each node is connected to at most one other node and isolated nodes are allowed.  Below is an example of a diagram in $\RBr_5$.\\

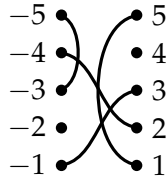
\begin{figure}[h!]
\centering
\begin{tikzpicture}
\fill (0,0)           circle (.75mm) node[left=2pt](a5) {$-5$};
\fill (0,\nodeheight) circle (.75mm) node[left=2pt](a4) {$-4$};
\multiply\nodeheight by 2
\fill (0,\nodeheight) circle (.75mm) node[left=2pt](a3) {$-3$};
\divide\nodeheight by 2
\multiply\nodeheight by 3
\fill (0,\nodeheight) circle (.75mm) node[left=2pt](a2) {$-2$};
\divide\nodeheight by 3
\multiply\nodeheight by 4
\fill (0,\nodeheight) circle (.75mm) node[left=2pt](a1) {$-1$};
\divide\nodeheight by 4

\fill (\nodewidth,0) circle (.75mm) node[right=2pt](b5) {$5$};
\fill (\nodewidth,\nodeheight) circle (.75mm) node[right=2pt](b4) {$4$};
\multiply\nodeheight by 2
\fill (\nodewidth,\nodeheight) circle (.75mm) node[right=2pt](b3) {$3$};
\divide\nodeheight by 2
\multiply\nodeheight by 3
\fill (\nodewidth,\nodeheight)  circle (.75mm) node[right=2pt](b2) {$2$};
\divide\nodeheight by 3
\multiply\nodeheight by 4
\fill (\nodewidth,\nodeheight)  circle (.75mm) node[right=2pt](b1) {$1$};
\divide\nodeheight by 5

\draw[e] (a1) to[out=0, in=180] (b3);
\draw[e] (a4) to[out=0, in=180]   (b2);
\draw[e] (a5) to[out=0, in=0] (a3);
\draw[e] (b1) to[out=180,in=180] (b5);

\end{tikzpicture}
\caption{Visualization of the partition $\{\{-5,-3\},\{-4,2\},\{-1,3\},\{5,1\},\{-2\},\{4\} \}$}
\label{fig:P5ex}
\end{figure}

Multiplying two diagrams, $\alpha$ and $\beta$, involves placing them side by side and identifying their middle nodes to create a new diagram, $\gamma$. In this process, any loop that appears in the middle is replaced by a factor of $\delta$, and any contractible component in the middle is replaced by a factor of $\epsilon$.\\

The {\it Motzkin Algebras} $\M_n = \M_n(R,\delta, \epsilon)$ is the $R$-subalgebra of $\RBr_n$ with basis given by planar Rook-Brauer diagrams i.e Rook-Brauer diagrams where connections cannot cross.  \\
\indent Diagrams in which every node on the left is connected to a single node on the right are called \emph{permutation diagrams}, and are in bijection with elements of the symmetric group $\fS_n$.  This gives us an inclusion and projection of algebras\\
\centerline{$\iota:R \fS_n \to \RBr_n$ \hspace{0.3in} and \hspace{0.3in}   $\pi: \RBr_n \to R \fS_n$}
where $\iota$ sends permutation diagrams to permutation diagrams and $\pi$ does the reverse while sending all remaining diagrams to $0$.  Note that $\pi \circ \iota$ is the identity map on $R \fS_n$. \\
\indent For each $n$, $\RBr_n$ and $\M_n$ are augmented algebras equipped with the augmented map that sends permutation diagrams to $1 \in R$ and all other diagrams to $0 \in R$.  This, in turn, makes $R$ a trivial module of $\RBr_n$ and $\M_n$ which we denoted by $\t$.  By \emph{homology} of Rook-Brauer algebras $\RBr_n$ and Motzkin algebras $\M_n$, we mean the Tor groups $\Tor_\ast^{\RBr_n}(\t,\t)$ and $\Tor_\ast^{\M_n}(\t,\t)$, respectively.\\
\indent The inclusion map $\iota:R \fS_n \to \RBr_n$ and projection map $\pi: \RBr_n \to R \fS_n$ also induce the corresponding maps $\iota_\ast$ and $\pi_\ast$ on homology groups such that $\pi_\ast\circ\iota_\ast$ is the identity.  Hence, homology of $\fS_n$ is a direct summand of the homology of $\RBr_n$.\\
We state our main results:
\begin{theorem}\label{main.thm}
    Suppose that $\epsilon$ is invertible in $R$ and for any $\delta \in R$, the inclusion map $\iota\colon R\mathfrak S_n \to \RBr_n(R,\delta,\epsilon)$ induces a map in homology \\
 \centerline{$\iota_\ast\colon H_\ast(\mathfrak S_n;\t) \longrightarrow \Tor^{\RBr_n(R,\delta,\epsilon)}_\ast(\t,\t)$}
that is an isomorphism for all degrees $\ast$.
\end{theorem}
\noindent Under the same hypothesis, the homology of the Motzkin algebra vanishes in positive degrees. 
\begin{theorem} \label{homology.Motzkin}
    Suppose that $\epsilon$ is invertible in $R$ and for any $\delta \in R$, \\
    \centerline{$\Tor^{\M_n(R,\delta,\epsilon)}_\ast(\t,\t) \cong \begin{cases} = \t~,~\ast = 0 \\ = 0~,~\ast > 0
        \end{cases}$}
\end{theorem}
Combining these theorems with the homological stability of symmetric groups proved by Nakaoka in \cite{nak} in which the stable range is sharp yields the following corollary. 
\begin{corollary}
    Suppose that $\epsilon$ is invertible in $R$ and for any $\delta \in R$, the Rook-Brauer algebras satisify homological stability i.e the inclusion $\RBr_{n-1}(R,\delta, \epsilon) \hookrightarrow \RBr_n(R,\delta,\epsilon)$ induces a map \\
    \centerline{$\Tor^{\RBr_{n-1}(R,\delta,\epsilon)}_i(\t,\t)
	\longrightarrow
	\Tor^{\RBr_{n}(R,\delta,\epsilon)}_i(\t,\t)$}
    that is an isomorphism in degrees $n\geq 2i+1$, and this stable range is sharp. \\
    Under the same assumption, the Motzkin algebras also satisfies homological stability.  
\end{corollary}

\section{Rook-Brauer and Motzkin Algebras}
In this section, we recall the definition of the Rook-Brauer and Motzkin algebras, certain diagrams and results that will be used in later sections. 
\subsection{Rook-Brauer algebras}
\begin{definition}
    For a unital commutative ring $R$, the {\it Rook-Brauer algebra} $\RBr_n = \RBr_n(R, \delta,\epsilon)$ with parameters $\delta, \epsilon \in R$ is a free $R$-module with basis consisting of partitions of the unions of the sets $[-n] \cup [n]$ whose blocks (components) have size $\le 2$ where $[-n] = \{-n,\ldots,-1\}$ and $[n] = \{1,\ldots,n\}$. Each basis element of $\RBr_n$ may be visualized by a diagram $\alpha$ with two vertical columns of $n$ nodes each, with the nodes on the left representing $-n$ through $-1$ and the nodes on the right representing $1$ through $n$, with paths connecting certain nodes.  Since each block has size at most 2, each basis element is represented by a diagram in which each node is connected to at most one other node and isolated nodes are allowed. 
\end{definition}
For diagrams $\alpha$ and $\beta$, we define the product $\alpha \beta$ in the following way. First we conjoin the diagrams $\alpha$ and $\beta$ by identifying each node $k$ on the right of $\alpha$ with the node $-k$ on the left of $\beta$, thus forming a diagram with $3$ columns of $n$ nodes each. We then form a new partition $\gamma$ of $[-n] \cup [n]$ in which two elements belong to the same block if and only if the corresponding nodes in the left or right column of the conjoined diagram are connected. We define $\alpha \beta = \delta^r \epsilon^s \cdot \gamma$ where $r$ is the number of loops and $s$ is the number of contractible components in the middle column (see example below).

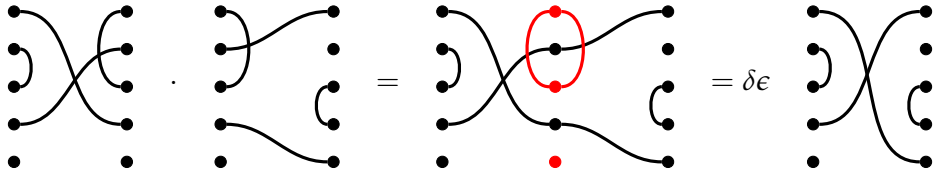
\begin{figure}[h!]
\centering
\begin{tikzpicture}[x=1.5cm,y=-.5cm,baseline=-1.05cm]

\node[v] (a1) at (0,0) {};
\node[v] (a2) at (0,1) {};
\node[v] (a3) at (0,2) {};
\node[v] (a4) at (0,3) {};
\node[v] (a5) at (0,4) {};

\node[v] (b1) at (1,0) {};
\node[v] (b2) at (1,1) {};
\node[v] (b3) at (1,2) {};
\node[v] (b4) at (1,3) {};
\node[v] (b5) at (1,4) {};

\draw[e] (a1) to[out=0, in=180] (b4);
\draw[e] (a2) to[out=0, in=0] (a3);
\draw[e] (a4) to[out=0, in=180]   (b2);
\draw[e] (b1) to[out=180, in=180] (b3);

\end{tikzpicture}
\quad
$\cdot$
\quad
\begin{tikzpicture}[x=1.5cm,y=-.5cm,baseline=-1.05cm]

\node[v] (b1) at (0,0) {};
\node[v] (b2) at (0,1) {};
\node[v] (b3) at (0,2) {};
\node[v] (b4) at (0,3) {};
\node[v] (b5) at (0,4) {};

\node[v] (c1) at (1,0) {};
\node[v] (c2) at (1,1) {};
\node[v] (c3) at (1,2) {};
\node[v] (c4) at (1,3) {};
\node[v] (c5) at (1,4) {};

\draw[e] (b1) to[out=0,in=0] (b3);
\draw[e] (b2) to[out=0,in=180] (c1);
\draw[e] (b4) to[out=0,in=180] (c5);
\draw[e] (c3) to[out=180,in=180] (c4);

\end{tikzpicture}
\quad
$=$
\quad
\begin{tikzpicture}[x=1.5cm,y=-.5cm,baseline=-1.05cm]

\node[v] (a1) at (0,0) {};
\node[v] (a2) at (0,1) {};
\node[v] (a3) at (0,2) {};
\node[v] (a4) at (0,3) {};
\node[v] (a5) at (0,4) {};

\node[red,v] (b1) at (1,0) {};
\node[v] (b2) at (1,1) {};
\node[red,v] (b3) at (1,2) {};
\node[v] (b4) at (1,3) {};
\node[red,v] (b5) at (1,4) {};

\node[v] (c1) at (2,0) {};
\node[v] (c2) at (2,1) {};
\node[v] (c3) at (2,2) {};
\node[v] (c4) at (2,3) {};
\node[v] (c5) at (2,4) {};

\draw[e] (a1) to[out=0, in=180] (b4);
\draw[e] (b4) to[out=0, in=180] (c5);
\draw[e] (a2) to[out=0, in=0]   (a3);
\draw[e] (a4) to[out=0, in=180] (b2);
\draw[e] (b2) to[out=0, in=180] (c1);
\draw[red,e] (b1) to[out=180,in=180] (b3);
\draw[red,e] (b1) to[out=0,in=0] (b3);
\draw[e] (c3) to[out=180,in=180] (c4);

\end{tikzpicture}
\quad
$ = \delta \epsilon$
\quad
\begin{tikzpicture}[x=1.5cm,y=-.5cm,baseline=-1.05cm]

\node[v] (a1) at (0,0) {};
\node[v] (a2) at (0,1) {};
\node[v] (a3) at (0,2) {};
\node[v] (a4) at (0,3) {};
\node[v] (a5) at (0,4) {};

\node[v] (c1) at (1,0) {};
\node[v] (c2) at (1,1) {};
\node[v] (c3) at (1,2) {};
\node[v] (c4) at (1,3) {};
\node[v] (c5) at (1,4) {};

\draw[e] (a1) to[out=0, in=180] (c5);
\draw[e] (a2) to[out=0, in=0] (a3);
\draw[e] (a4) to[out=0, in=180]  (c1);
\draw[e] (c3) to[out=180,in=180] (c4);

\end{tikzpicture}
\caption{Multiplication in $\RBr_5(\delta,\epsilon)$.}
\label{fig:MultP5}
\end{figure}

\noindent Here, the final diagram is multiplied by a factor of $\delta$ for the red loop and a factor of $\epsilon$ for the isolated node in the middle.\\
\indent We also point out here three specific types of diagram (see Figure \ref{three.types}) that will be used later: 
\begin{itemize}
    \item For~$1\leq i\leq n-1$, $S_i$ is the permutation diagram corresponding to the partition with blocks of pairs~$\{-j,j\}$ for~$j\neq i, i+1$, together with~$\{-(i+1),i\}$ and $\{-i,(i+1)\}$. These generate the group ring of the symmetric group, $\fS_n$, as a subalgebra of~$\RBr_n$.
    \item For~$1\leq i\neq j\leq n-1$, $V_{ij}$ is the diagram corresponding to the partition with blocks of pairs~$\{-k,k\}$ for~$k\neq i, j$ and two blocks consists of $\{-i,-j\}$ and $\{i,j\}$.
    \item For~$1\leq i\leq n$, $T_i$ is the diagram corresponding to the partition with blocks of pairs~$\{-j,j\}$ for~$j\neq i$ and two singleton blocks~$\{-i\}$ and~$\{i\}$.
\end{itemize}
 
\begin{figure}[h!] \label{three.types}
    \centering
    \[
    \begin{tikzpicture}[x=1.5cm,y=-.5cm,baseline=-1.05cm]
    
    \node[v] (a1) at (0,0) {};
    \node[v] (a2) at (0,1) {};
    \node[v] (a3) at (0,2) {};
    \node[v] (a4) at (0,3) {};
    
    \node[v] (b1) at (1,0) {};
    \node[v] (b2) at (1,1) {};
    \node[v] (b3) at (1,2) {};
    \node[v] (b4) at (1,3) {};
    
    \draw[e] (a1) to[out=0, in=180] (b1);
    \draw[e] (a2) to[out=0, in=180] (b3);
    \draw[e] (a3) to[out=0, in=180]   (b2);
    \draw[e] (a4) to[out=0, in=180] (b4);

    \node at (0.5,4) {$S_2$};
    \end{tikzpicture}
    \qquad \qquad
    \begin{tikzpicture}[x=1.5cm,y=-.5cm,baseline=-1.05cm]
    
    \node[v] (a1) at (0,0) {};
    \node[v] (a2) at (0,1) {};
    \node[v] (a3) at (0,2) {};
    \node[v] (a4) at (0,3) {};
    
    \node[v] (b1) at (1,0) {};
    \node[v] (b2) at (1,1) {};
    \node[v] (b3) at (1,2) {};
    \node[v] (b4) at (1,3) {};
   
    \draw[e] (a1) to[out=0, in=180] (b1);
    \draw[e] (a3) to[out=0, in=180] (b3);
    \draw[e] (a2) to[out=0, in=0] (a4);
    \draw[e] (b2) to[out=180, in=180] (b4);
    
    \node at (0.5,4) {$V_{13}$};
    \end{tikzpicture}
    \qquad\qquad
    \begin{tikzpicture}[x=1.5cm,y=-.5cm,baseline=-1.05cm]
    
    \node[v] (a1) at (0,0) {};
    \node[v] (a2) at (0,1) {};
    \node[v] (a3) at (0,2) {};
    \node[v] (a4) at (0,3) {};
    
    \node[v] (b1) at (1,0) {};
    \node[v] (b2) at (1,1) {};
    \node[v] (b3) at (1,2) {};
    \node[v] (b4) at (1,3) {};
   
    \draw[e] (a1) to[out=0, in=180] (b1);
    \draw[e] (a3) to[out=0, in=180] (b3);
    \draw[e] (a4) to[out=0, in=180] (b4);
    
    \node at (0.5,4) {$T_3$};
    \end{tikzpicture}
    \]
    \caption{The elements $S_2,V_{13},T_3\in\RBr_4$}
\end{figure}
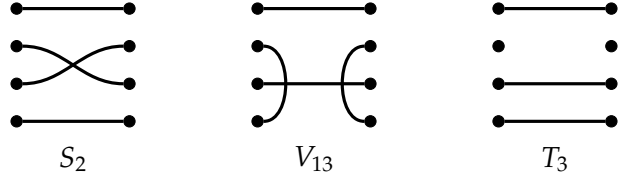

Recall that permutation diagrams of $\RBr_n$ are diagrams where each node on the left is connected to a node on the right i.e there are no isolated nodes or same-side connections in these diagram.
\begin{definition} \label{trivial.module}
    (Trivial module $\t$) For any $n$, we define the \emph{trivial} $\RBr_n$-module $\t$ to be a single copy of $R$ where permutation diagrams acts as the identity and all other diagrams acts as zero.  
\end{definition}

\begin{definition}
    For~$m\leq n$, we can view~$\RBr_m$ as a subalgebra of~$\RBr_n$ as follows: given a diagram $\alpha$ in $\RBr_m$, we can add $n-m$ horizontal connections \emph{below} $\alpha$ to form a diagram in $\RBr_n$.
    Then, under the action of this subalgebra,~$\RBr_n$ can be viewed as a left $\RBr_n$-module and a right~$\RBr_m$-module, and we obtain the induced left $\RBr_n$-module~$\RBr_n \otimes_{\RBr_m} \t$.
\end{definition}

It turns out that the left $\RBr_n$-module $\RBr_n \otimes_{\RBr_m} \t$ is a free $R$-module whose basis are given in terms of special diagrams as described by the following Proposition from \cite{mt}.
\begin{proposition}
    (\cite{mt}, Proposition 3.1) $\RBr_n \otimes_{\RBr_m} \t$ is a free $R$-module with a basis consisting of diagrams with $n$ nodes on the left, $m$ nodes on the right under a box containing the top $n-m$ nodes subject to these conditions: 
    \begin{itemize}
        \item Each node is either connected to the box, to another node, or remains isolated. The box must be connected to exactly $n-m$ nodes.
        \item Any diagram in which the box is connected to fewer than $n-m$ nodes is identified with 0.
    \end{itemize}
\end{proposition}
Some (non-)examples of these basis diagrams in $\RBr_5 \otimes_{\RBr_2} \t$ are given below. 
\begin{figure}[h!]
\centering
\begin{tikzpicture}[x=1.5cm,y=-.5cm,baseline=-1.05cm]

\node[v] (b1) at (0,0) {};
\node[v] (b2) at (0,1) {};
\node[v] (b3) at (0,2) {};
\node[v] (b4) at (0,3) {};
\node[v] (b5) at (0,4) {};

\node[v] (c1) at (1,0) {};
\node[v] (c2) at (1,1) {};
\node[v] (c3) at (1,2) {};
\node[v] (c4) at (1,3) {};
\node[v] (c5) at (1,4) {};

\draw[e] (b2) to[out=0,in=0] (b5);
\draw[e] (b1) to[out=0,in=180] (c1);
\draw[e] (b3) to[out=0, in=180] (c2);
\draw[e] (c4) to[out=180,in=180] (c3);

\draw[e,fill=white] (0.9,-0.2) rectangle (1.1,2.2);
\node at (c2) {$3$};
\end{tikzpicture}
\quad
,
\quad
\begin{tikzpicture}[x=1.5cm,y=-.5cm,baseline=-1.05cm]

\node[v] (b1) at (0,0) {};
\node[v] (b2) at (0,1) {};
\node[v] (b3) at (0,2) {};
\node[v] (b4) at (0,3) {};
\node[v] (b5) at (0,4) {};

\node[v] (c1) at (1,0) {};
\node[v] (c2) at (1,1) {};
\node[v] (c3) at (1,2) {};
\node[v] (c4) at (1,3) {};
\node[v] (c5) at (1,4) {};

\draw[e] (b3) to[out=0,in=0] (b5);
\draw[e] (b4) to[out=0,in=180] (c1);
\draw[e] (c4) to[out=180, in=180] (c2);
\draw[e] (c5) to[out=180,in=180] (c3);

\draw[e,fill=white] (0.9,-0.2) rectangle (1.1,2.2);
\node at (c2) {$3$};
\end{tikzpicture}
\quad
,
\quad
\begin{tikzpicture}[x=1.5cm,y=-.5cm,baseline=-1.05cm]

\node[v] (b1) at (0,0) {};
\node[v] (b2) at (0,1) {};
\node[v] (b3) at (0,2) {};
\node[v] (b4) at (0,3) {};
\node[v] (b5) at (0,4) {};

\node[v] (c1) at (1,0) {};
\node[v] (c2) at (1,1) {};
\node[v] (c3) at (1,2) {};
\node[v] (c4) at (1,3) {};
\node[v] (c5) at (1,4) {};

\draw[e] (b3) to[out=0,in=0] (b5);
\draw[e] (b2) to[out=0,in=180] (c1);
\draw[e] (b1) to[out=0, in=180] (c4);
\draw[e] (c5) to[out=180,in=180] (c3);

\draw[e,fill=white] (0.9,-0.2) rectangle (1.1,2.2);
\node at (c2) {$3$};
\end{tikzpicture}
\quad
$= 0 \in \RBr_5 \otimes_{\RBr_2} \t$
\end{figure}

\subsection{Motzkin algebras} 
\begin{definition}
    A \emph{Motzkin $n$-diagram} is a planar Rook-Brauer $n$-diagram i.e diagrams where connections cannot cross.  The \emph{Motzkin algebras}, $\M_n = \M_n(R,\delta,\epsilon)$, is an $R$-subalgebra of $\RBr_n(R,\delta,\epsilon)$ with basis given by all Motzkin $n$-diagrams. 
\end{definition}
Note that, due to the restriction of planarity, the only permutation diagram in $\M_n$ is the \emph{identity} diagram i.e diagram with all horizontal left-to-right connections.  Since $\M_n$ is a subalgebra of $\RBr_n$, we also have the \emph{trivial} $\M_n$-module $\t$ as in Def.~  
\ref{trivial.module}.\\
We recall the definition of link states for a Motzkin diagram below. Although this concept can be defined more generally for Rook–Brauer diagrams, our focus here is solely on the Motzkin diagram.
\begin{definition}
    By slicing vertically down the middle of a Motzkin $n$-diagram, we obtain two "half-diagrams" which are called \emph{left link state} and \emph{right link state} of the diagram.  Explicitly, a link state consists of a column of $n$ nodes where at each node, we have one of the following situations: 
    \begin{itemize}
        \item The node has a hanging edge called a \emph{defect}  i.e an edge whose other end is not attached to anything. 
        \item The node is connected to exactly one other node.
        \item The node is an isolated node.
    \end{itemize}
\end{definition}
Since a Motzkin diagram is a planar diagram, the right and left link state of it is also planar.  Below is an example of a Motzkin diagram and its right link state. 

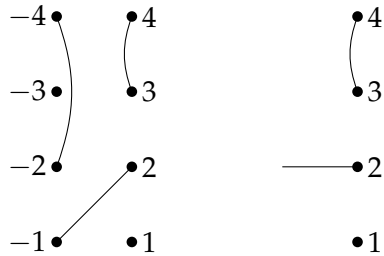
\begin{figure}[h!]

\begin{tikzpicture}
\fill (0,0) circle[radius=2pt];
\fill (0,1) circle[radius=2pt];
\fill (0,2) circle[radius=2pt];
\fill (0,3) circle[radius=2pt];
\fill (1,0) circle[radius=2pt];
\fill (1,1) circle[radius=2pt];
\fill (1,2) circle[radius=2pt];
\fill (1,3) circle[radius=2pt];   
\fill (4,0) circle[radius=2pt];
\fill (4,1) circle[radius=2pt];
\fill (4,2) circle[radius=2pt];
\fill (4,3) circle[radius=2pt];

\draw (0,0) node[left] {$-1$};
\draw (0,1) node[left] {$-2$};
\draw (0,2) node[left] {$-3$};
\draw (0,3) node[left] {$-4$};
\draw (1,0) node[right] {$1$};
\draw (1,1) node[right] {$2$};
\draw (1,2) node[right] {$3$};
\draw (1,3) node[right] {$4$};
\draw (4,0) node[right] {$1$};
\draw (4,1) node[right] {$2$};
\draw (4,2) node[right] {$3$};
\draw (4,3) node[right] {$4$};

\draw (0,0) -- (1,1);
\path[-] (0,3) edge [bend left=20] (0,1);
\path[-] (1,2) edge [bend left=20] (1,3);
\draw (4,1) -- (3,1);
\path[-] (4,2) edge [bend left=20] (4,3);
\end{tikzpicture}

\caption{A Motzkin $4$-diagram and its right link state.}
\end{figure}

\begin{definition} \label{splices.deletion}
    Given a right link state of a Motzkin $n$-diagram, we can remove two defects and replace them with a right-to-right connection joining the two vertices. This operation is called a \emph{splice}. We can also remove a defect, leaving an isolated vertex. This operation is called a \emph{deletion}.  Note that all splice operations must respect the planarity condition.  
\end{definition}

\section{Inductive Resolution}
A key tool in many proofs of the homological stability of groups is Shapiro's Lemma. However, when dealing with algebras that do not arise as group algebras, Shapiro's Lemma does not always apply  as the algebra $A$ may fail to be flat over its subalgebras (see sect 4 of \cite{bhp1} and sect 1 of \cite{bh}).  To address this issue, we will use \emph{inductive resolution} prove an analogue (more general) version of Shapiro's Lemma for the Rook-Brauer Algebra. \\
Simply put, the technique of inductive resolution provides a method for proving the vanishing of $\Tor^A_{\ast}(M,-)$ of a given module by resolving it through modules that already possess this property- hence the term "inductive" resolution. We formally state the principle of inductive resolution as introduced in \cite{bhp}.
\begin{theorem}\label{inductive-resolution-method}
    (\cite{bhp}, Theorem 3.1) Let $A$ be an algebra over a ring $R$, and let $M$ be a right $A$-module.
    Suppose that $N$ is a left $A$-module equipped with a resolution $Q_\ast\to N$ with the following two properties:
    \begin{itemize}
        \item
        $\Tor^A_\ast(M,Q_j)$ vanishes in positive degrees for all $j\geqslant 0$.
        \item
        $M\otimes_A Q_\ast\to M\otimes_A N$ is a resolution. 
    \end{itemize}
    Then $\Tor^A_\ast(M,N)$ vanishes in positive degrees.
\end{theorem}
For the rest of this section, we will use the technique of inductive resolution to prove $\Tor$ of certain modules vanishes for positive degrees.

\subsection{Inductive Resolution for Rook-Brauer Algebras} \label{section.for.Rook.Brauer}
\begin{definition} \label{def.of.J_X}
    Suppose that $X$ is a subset of the set $\{1,\dots,n\}$.
    Define $J_X$ to be the left-ideal in $\RBr_n$ that is the $R$-span of all diagrams in which, among the nodes on the right labeled by elements of $X$, there is at least one singleton or a right-to-right connection. For $m\le n$, let $J_{\{n-m+1,\dots,n\}}$ be denoted by $J_m$.
\end{definition}
Notice that $J_n \subseteq \RBr_n$ is spanned by all non-permutation diagrams of $\RBr_n$ so $\RBr_n/J_n \cong R\mathfrak S_n$. We will use inductive resolution to establish the following theorem, which will play a key role in proving the analogues of Shapiro's Lemma.
\begin{theorem}\label{vanishing.RBr}
    Let $X\subseteq\{1,\dots,n\}$ and suppose $\epsilon \in R$ is invertible and $\delta \in R$, then the groups $\Tor_\ast^{\RBr_n}(\t,\RBr_n/J_X)$ vanish in positive degrees.
\end{theorem}
The proof of this result will occupy the rest of this section and we follow a similar outline as section 4 of \cite{bhp} with some modifications.  Below, we prove a short lemma that is necessary for Theorem \ref{thm.for.inductive}.
\begin{lemma}\label{tensor.RBr/J}
     Let $J$ be a left ideal of $\RBr_n$ that is included in $J_n$. Then\\
     \centerline{$\t \otimes_{\RBr_n} \RBr_n/J \cong \t.$}
    In particular,\\
    \centerline{$ \Tor^{\RBr_n}_0(\t,\RBr_n/J_X) \cong \t $}
    for all $X\subset \{1, \dots, n\}$.
\end{lemma}
\begin{proof}
    Since $J \subseteq J_n$, elements of $J$ acts as zero on $\t$ and this gives $\t \otimes_{\RBr_n} \RBr_n/J \cong \t.$
\end{proof}
We will prove Theorem \ref{thm.for.inductive} by induction on the cardinality of $|X|$ and to do that, we will resolve $\RBr_n/J_X$ in term of these special modules introduced below. 
\begin{definition} \label{def.of.Ax.Bx}
    Let $\{a,b\} \subseteq X \subseteq\{1,\dots,n\}$, and let $x\in X$, define three left $\RBr_n$-submodules of $\RBr_n$ as follows:
    \begin{itemize}
        \item
        $A_x$ is the span of all diagrams in which $x$ is a singleton.
        \item
        $B_{X,x}$ is the span of all diagrams in which $x$ is connected to some element of $X$.
        \item
        $M_{\{a,b\}}$ is the span of all diagrams in which nodes $a$ and $b$ are connected.
    \end{itemize}
    We also define the quotients:
    \begin{gather*}
        \calA_{X,x}=\frac{A_x}{A_x\cap J_{X-\{x\}}},
        \qquad
        \calB_{X,x} = \frac{B_{X,x}}{B_{X,x}\cap J_{X-\{x\}}},
        \qquad
        \calM_{X,\{a,b\}}=\frac{M_{\{a,b\}}}{M_{\{a,b\}} \cap J_{X-\{a,b\}}}
    \end{gather*}
\end{definition}
We prove some results about these special modules. 
\begin{lemma}\label{modules.lemma}
    Let $\epsilon \in R$ be invertible.  The modules $\calA_{X,x}$, $\calB_{X,x}$, and $\calM_{X,Y}$ behave as follows under tensor product with $\t$.
    \begin{itemize}
        \item
        Let $x\in X\subseteq\{1,\dots,n\}$.
        Then $\t\otimes_{\RBr_n}\calA_{X,x}=0$ and $\calA_{X,x}$ is a direct summand of $\RBr_n/J_{X-\{x\}}$.
        \item
        Let $x\in X\subseteq\{1,\dots,n\}$ with $n\geq 2$.
        Then $\t\otimes_{\RBr_n}\calB_{X,x}=0$.
        \item
         Let $a, b \in X$, then $\t\otimes_{\RBr_n}\calM_{X,\{a,b\}}=0$.  Furthermore, $\calM_{X,\{a,b\}}$ is a direct summand of $\RBr_n/J_{X-\{a,b\}}$.
    \end{itemize}
\end{lemma}
\begin{proof}
    Since $\calA_{X,x}, \calB_{X,x}, \calM_{X,\{a,b\}}$ are quotients of $A_x, B_{X,x}$ and $M_{\{a,b\}}$ respectively, we show that these modules vanish when tensoring  with $\t$.\\
    \indent To show $\t \otimes_{\RBr_n} A_x = 0$, let $\alpha$ be a diagram in $A_x$  which means node $x$ is a singleton.  This means $\alpha$ also has another singleton on the left or right side, say $x'$.  If $x'$ is on the left, then $\alpha = T_{x'}\alpha'$ and if $x'$ is on the right, then $\alpha = \alpha' T_x$ where in both cases, $\alpha'$ is the diagram obtained from $\alpha$ by connecting $x$ and $x'$ together while leaving all other connections unchanged.  Then, in both cases, \\
    \centerline{$1 \otimes \alpha = 1 \otimes T_{x'}\alpha' = 1\cdot T_{x'} \otimes \alpha' = 0$ \hspace{0.2in}and \hspace{0.2in}$1 \otimes \alpha = 1 \otimes \alpha' T_{x} = 1\cdot \alpha' \otimes T_{x} = 0$}
    because $T_{x'}$ and $\alpha'$ in the second case are both non-permutation diagrams and hence, acts as 0 on $\t$.\\  For the second part, we prove that $\calA_{X,x}$ is a direct summand of $\RBr_n/J_{X-\{x\}}$.  Since right-multiplying by $\epsilon^{-1}T_x$ takes $J_{X-\{x\}}$ to itself, this induces the map $\RBr_n/J_{X-\{x\}} \xrightarrow[]{\epsilon^{-1}T_x} \calA_{X,x}$.  This map is surjective because for any diagram $\alpha \in \calA_{X,x}$, we have $\alpha T_x = \epsilon \alpha$; it also splits because $\epsilon^{-1} T_x$ is idempotent with splitting map $\calA_{X,x} \to \RBr_n/J_{X-\{x\}}$ induced by the inclusion $A_x \hookrightarrow \RBr_n$. \\
    \indent To show $\t \otimes_{\RBr_n} B_{X,x} = 0$, let $\alpha \in B_{X,x}$ be a diagram, then the node $x$ is connected to some node $y$ of $X$.  Since each node is connected to at most one other node in any Rook-Brauer diagram, this implies that there is either a left-to-left connection or a pair of isolated nodes on the left. \\
    \indent Case 1:  If there is a left-to-left connection from $x'$ to $y'$, choose $\alpha'$ to be the diagram obtained from $\alpha$ by removing the connection from $x$ to $y$.  Notice that $\alpha = \epsilon^{-1} \alpha' V_{xy}$.\\
    \indent Case 2: If there is a pair of isolated nodes $\{x',y'\}$ on the left, choose $\alpha'$ to be the diagram obtained from $\alpha$ by removing the connection from $x$ to $y$, then either connect $x'$ with $x$ or $y'$ with $y$ but not both.  Notice that $\alpha = \alpha' V_{xy}$.   \\
    In both case, $\alpha'$ is a non-permutation diagram which gives $\t \otimes \alpha = 0$.\\
    \indent Finally, we conclude that $ \t \otimes M_{\{a,b\}} = 0$  by applying the second part above.  To see that $\calM_{X,\{a,b\}}$ is a direct summand of $\RBr_n/J_{X-\{a,b\}}$, note that right-multiplication by $T_a V_{ab}$ takes $J_{X-\{a,b\}}$ into itself so this induces the map \\
    \centerline{$\RBr_n/J_{X-\{a,b\}} \xrightarrow[]{\epsilon^{-1}T_a V_{ab}} \calM_{X,\{a,b\}}$,}
    we will show that this map is surjective and splits.  To see that the map is surjective, pick any diagram $\alpha \in \calM_{X,\{a,b\}}$, we then have $\alpha T_a V_{ab} = \epsilon \alpha$ which shows that the map is surjective. \\
    This map also splits because $\epsilon^{-1}T_a V_{ab}$ is idempotent with the splitting map $\iota:\calM_{X,\{a,b\}} \to \RBr_n/J_{X-\{a,b\}}$ induced by the inclusion $M_{\{a,b\}} \hookrightarrow \RBr_n$.
\end{proof}
One can readily verify that $\calB_{X,x}$ decomposes as a direct sum of $\calM_{X,\{a,b\}}$'s as follows. 
\begin{lemma} \label{lemma.for.BX.x}
    For any $\delta, \epsilon \in R$, there exists a left $\RBr_n$-modules isomorphism $\underset{x_o \in X}{\bigoplus} \calM_{X,\{x,x_o\}} \cong \calB_{X,x}$.
\end{lemma}
\begin{proof}
    The isomorphism is given by $\iota:\underset{x_o \in X}{\bigoplus} \calM_{X,\{x,x_o\}} \to \calB_{X,x}$ where, on each summand, $\iota$ is the map $ \calM_{X,\{x,x_o\}} \to \calB_{X,x}$ induced by the inclusions \\
    \centerline{$M_{\{x,x_o\}} \hookrightarrow \calB_{X,x}$ and $J_{X-\{x,x_o\}} \hookrightarrow J_{X-\{x\}}$.}
    To see that $\iota$ is surjective, observe that for any diagram $\alpha \in \calB_{X,x}$, the node $x$ must be connected to some node $x_o \in X$. This means $\alpha$ lies in the image of the direct summand $\calM_{X,\{x,x_o\}}$ indexed by $x_o \in X$, which shows that $\iota$ is indeed surjective.\\
    Note that images of diagrams from different direct summand are distinct because for $x_o \neq x_o'$, diagrams in $M_{\{x,x_o\}}$ has node $x$ connected with $x_o$ while diagrams in $M_{\{x,x_o'\}}$ has node $x$ connected with $x_o'$.  Hence, to show that $\iota$ is injective, we can show for each direct summand,  $\calM_{X,\{x,x_o\}} \to \calB_{X,x}$ is injective.  Pick a diagram $\alpha$ in the kernel of this map i.e $\alpha = 0$ in $\calB_{X,x}$.  This implies $\alpha$ lies in $J_{X-\{x\}}$ which means, two nodes of $\alpha$ labeled by elements of $X-\{x\}$ must be connected or a node labeled by $X-\{x\}$ is isolated.  \\
    Since $\alpha$ lies in $M_{\{x,x_o\}}$, $x_o$ is not  isolated and is connected with $x$.  Hence, $\alpha$ must have a right-to-right connection between two nodes labeled by $X-\{x,x_o\}$ or an isolated node among nodes labeled by $X-\{x,x_o\}$.  But this implies $\alpha \in J_{X-\{x,x_o\}}$ so this map is injective, hence $\iota$ is also injective. 
\end{proof}

We now resolve $\RBr_n/J_X$ in terms of $\calA_{X,x}$ and $\calB_{X,x}$.  The proof of the following proposition is almost identical, with small modifications, to the proof of Proposition 4.6 presented in \cite{bhp}.

\begin{proposition} \label{thm.for.inductive}
     Let $X\subseteq\{1,\dots,n\}$, let $x\in X$, and assume $n\geq 2$.
    The following sequence, in which all maps are induced by either an inclusion or an identity map, is a resolution of $\RBr_n/J_X$.
    \[\xymatrix@R=5pt{
        \dots
        \ar[r]
        &
        0
        \ar[r]
        &
        \calA_{X,x}
        \oplus
        \calB_{X,x}
        \ar[r]
        &
        \RBr_n/J_{X-\{x\}} 
        \ar[r]
        &
        \RBr_n/J_X
        \\
        &
        2
        &
        1
        &
        0
        &
        -1
    }\]
    Moreover, applying $\t\otimes_{\RBr_n}\!\!-$ to the sequence gives a resolution of $\t\otimes_{\RBr_n} \RBr_n/J_X$.
\end{proposition}
\begin{proof}
    Since all maps are induced by either inclusion or identity, all maps are well-defined because $J_{X-\{x\}} \subseteq J_X$, $(A_x \cap J_{X-\{x\}}) \subseteq J_{X-\{x\}}$ and $(B_{X,x} \cap J_{X-\{x\}}) \subseteq J_{X-\{x\}}$. \\
    The surjectivity of the map $ \RBr_n/J_{X-\{x\}} \to  \RBr_n/J_{X}$ is clear so the complex is exact at degree $-1$.  To see the complex is exact at degree $0$, we observe that the kernel of the map $ \RBr_n/J_{X-\{x\}} \to  \RBr_n/J_{X}$ is spanned by diagrams in $\RBr_n$ that lies in $J_X$.  For a diagram $\alpha \in \RBr_n$ to be in  $J_X$, there must be a right-to-right connection or an isolated node among nodes labeled by $X$.  For it to \emph{also} not be in $J_{X-\{x\}}$, among nodes labeled by $X$, it must have exactly one right-to-right connection between $x$ and another node in $X$ or $x$ must be the only isolated node.  If $\alpha$ has a right-to-right connection between $x$ and another node in $X$, then $\alpha$ is in the image of the inclusion $B_{X,x} \hookrightarrow \RBr_n$.  If $x$ is an isolated node in $\alpha$, then $\alpha$ is in the image of the inclusion $A_x \hookrightarrow \RBr_n$.  Hence, the complex is exact at degree $0$.\\
    To show exactness at degree $1$ i.e the map $\calA_{X,x} \oplus \calB_{X,x} \to  \RBr_n/J_{X-\{x\}}$ is injective, note that if $\alpha \in A_x$ and $\beta \in B_{X,x}$ such that $\alpha + \beta \in J_{X-\{x\}}$, then $\alpha, \beta \in J_{X-\{x\}}$ because $A_x$ and $B_{X,x}$ have no basis elements in common.  Hence, the complex is exact at degree $1$ and it is a resolution of $\RBr_n/J_X$.\\
    To prove the second claim, after applying $\t\otimes_{\RBr_n}\!\!-$ and by Lemmas \ref{modules.lemma} and \ref{tensor.RBr/J}, the resolution becomes 
    \[\xymatrix@R=5pt{
        \dots
        \ar[r]
        &
        0
        \ar[r]
        &
        0
        \ar[r]
        &
        \t
        \ar[r]^{\mathrm{Id}}
        &
        \t
        \\
        &
        2
        &
        1
        &
        0
        &
        -1
    }\]
    and the claim now follows.
\end{proof}

We use this proposition to prove the vanishing of Tor's group of $\RBr_n/J_X$, Theorem \ref{vanishing.RBr}.  The proof below follows a similar outline with some modifications to that of the analogous statement for Partition algebras. (see section 4.4 of \cite{bhp}).\\

\noindent {\bf Proof of Theorem \ref{vanishing.RBr}}
    For the case when $n = 0$, we have $\RBr_n = R$ and the result follows immediately.  For the case when $n = 1$, we either have $X = \emptyset$ or $X = \{1\}$.  If $X = \emptyset$, then $J_X = 0$ so that $\RBr_n/J_X = \RBr_n$ and the result follows.  If $X = \{1\}$, then $J_X$ is spanned by $T_1$ which implies $J_X = A_x$ and we have a SES   \\
    \[\xymatrix@R=5pt{
        0
        \ar[r]
        &
        \calA_{X,x}
        \ar[r]
        &
        \RBr_n
        \ar[r]
        &
        \RBr_n/J_X
        \ar[r]
        &
        0}. \]

    Since $\epsilon$ is invertible, $\calA_{X,x} = A_x$ and $\RBr_n/J_X$ are direct summand of $\RBr_n$ and the result now follows.  \\
    Assume that $n \geq 2$ and $X \subseteq [n]$.  We prove the theorem by using strong induction on cardinality of $X$.  From above, the result is clear when $X = \emptyset$.  Assume $|X| > 0$ and the result holds for any subset $X' \subseteq [n]$ of smaller cardinality. \\
    By Proposition \ref{thm.for.inductive} and Theorem \ref{inductive-resolution-method}, it suffices to show that the three modules $\calA_{X,x}, \calB_{X,x}$ and $\RBr_n/J_{X-\{x\}}$ all vanish under $\Tor^{\RBr_n}_i(\t,-)$ for $i>0$. \\ But this is immediate since $\Tor^{\RBr_n}_i(\t,\RBr_n/J_{X-\{x\}}) = 0$ for $i > 0 $ because of the induction hypothesis and $\Tor^{\RBr_n}_i(\t,\calA_{X,x}) = 0$ for $i > 0 $ because $\calA_{X,x}$ is a direct summand of $\RBr_n/J_{X-\{x\}}$ by Lemma \ref{modules.lemma}, which vanishes under $\Tor^{\RBr_n}_i(\t,-)$ so $\calA_{X,x}$ does as well. \\
    By Lemma \ref{lemma.for.BX.x}, $\calB_{X,x}$ is a direct summand of $\calM_{X,\{x,x_o\}}$'s for $x_o \in X$ and Lemma \ref{modules.lemma} implies $\calM_{X,\{x,x_o\}}$ is also a direct summand of $\RBr_n/J_{X-\{x,x_o\}}$ which vanishes under $\Tor^{\RBr_n}_i(\t,-)$ because of the induction hypothesis.  This, in turn, implies  $\calM_{X,\{x,x_o\}}$ also vanishes under $\Tor^{\RBr_n}_i(\t,-)$ and so does $\calB_{X,x}$.
    \qed

\subsection{Inductive Resolution for Motzkin Algebras} This subsection is similar to the subsection \ref{section.for.Rook.Brauer} above and we will prove the analogue of Theorem \ref{vanishing.RBr} for the Motzkin algebras.  \\
\indent By replacing $\RBr_n$ with $\M_n$ in Definition \ref{def.of.J_X} and \ref{def.of.Ax.Bx}, we obtain similar left $\M_n$-submodules of $\M_n$, namely $J_X$, $A_x$, $B_{X,x}$ and the corresponding quotients $\calA_{X,x}$, $\calB_{X,x}$.  We also introduce a new left $\M_n$-submodule tailored specifically to the structure of Motzkin algebras. \\

\vspace*{-0.2in}
\noindent {\bf Notational Remark:} For $a,b \in [n]$ with $a < b$, define $[a,b] = \{i \in [n]~|~a \leq i \leq b\}$. 
\begin{definition}
    Let $a,b \in [n]$ with $a < b$, if $P$ is a partition of $[a,b]$ into subsets of size at most 2, define $Y_P$ to be the right-link state of a Motzkin $n$-diagram in which all right-to-right connections and isolated nodes are in one-to-one correspondence to pairs and singletons, respectively, in $P$ (labeled by pairs and singletons in $P$) while all other nodes have defects.  \\If any of the above connections violate planarity, there is no right link state for that particular $P$.
\end{definition}
For example, if $n = 6, ~a = 1,~ b = 5$ and $P = \{\{1,5\},\{2,3\},\{4\}\}$, then $Y_P$ is given below. 
\begin{figure}[h!]
\centering
\begin{tikzpicture}
\fill (4,0) circle[radius=2pt];
\fill (4,1) circle[radius=2pt];
\fill (4,2) circle[radius=2pt];
\fill (4,3) circle[radius=2pt]; 
\fill (4,4) circle[radius=2pt]; 
\fill (4,5) circle[radius=2pt]; 

\draw (4,0) node[right] {$1$};
\draw (4,1) node[right] {$2$};
\draw (4,2) node[right] {$3$};
\draw (4,3) node[right] {$4$};
\draw (4,4) node[right] {$5$}; 
\draw (4,5) node[right] {$6$}; 
\node at (2,2.5) {$Y_P = $};

\draw (4,5) -- (3,5);
\path[-] (4,0) edge [bend left=25] (4,4);
\path[-] (4,1) edge [bend left=22] (4,2);

\end{tikzpicture}
\caption{$Y_P$ when $n = 6, ~a = 1,~ b = 5$ and $P = \{\{1,5\},\{2,3\},\{4\}\}$ }
\end{figure}
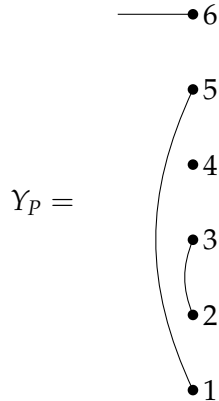

\begin{definition}
    Given a right-link state $Y_P$, define a left $\M_n$-submodule $\bY_P$ spanned by diagrams having right-link state obtained from $Y_P$ by a (possibly empty) sequence of splices and deletion (see Def   \ref{splices.deletion}).  If no right-link state exists for $P$, define $\bY_P$ to be the zero module. 
\end{definition}

\begin{definition}
    Given $a, b \in X \subseteq [n]$ with $a < b$, let $P$ be a partition of $[a,b]$ as above with $\{a,b\} \in P$ and the corresponding right-link state $Y_P$, define the quotient 
    $$\bY_{P,\{a,b\}} \coloneqq \dfrac{\bY_P}{\bY_P \cap J_{X-\{a,b\}}}$$
\end{definition}

We prove the analogue of Lemma \ref{modules.lemma} for $\bY_{P,\{x,y\}}$.
\begin{lemma}  \label{prop.Y_P}
    Let $x, y \in X$ and $P$ be a partition of $[x,y]$ if $x < y$ or $[y,x]$ if $y < x$ as above with $\{x,y\} \in P$, then $\t \otimes_{\M_n} \bY_{P,\{x,y\}} = 0$.  Furthermore, $\bY_{P,\{x,y\}}$ is a direct summand of $\dfrac{\M_n}{J_{X'}}$ where $X' \subseteq [n]$ with $|X'| < |X|$.
\end{lemma}
\begin{proof}
    To show $\t \otimes_{\M_n} \bY_{P,\{x,y\}} = 0$, we show $\t \otimes_{\M_n} \bY_{P} = 0$.  For a diagram $\alpha \in \bY_P$, define $\alpha'$ to be the diagram obtained from $\alpha$ by preserving all right-to-right connections of $\alpha$ and for the other nodes, make all horizontal left-to-right connections without violating planarity. 
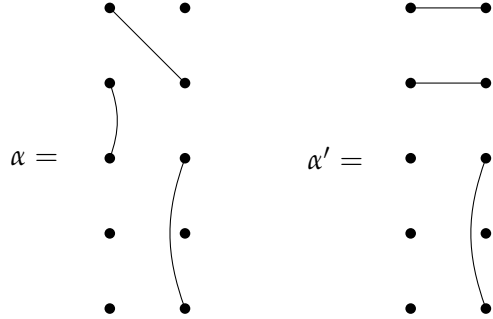
\begin{figure}[h!]

\begin{tikzpicture}
\node at (-1,2) {$\alpha =$};
\fill (0,0) circle[radius=2pt];
\fill (0,1) circle[radius=2pt];
\fill (0,2) circle[radius=2pt];
\fill (0,3) circle[radius=2pt];
\fill (0,4) circle[radius=2pt];
\fill (1,0) circle[radius=2pt];
\fill (1,1) circle[radius=2pt];
\fill (1,2) circle[radius=2pt];
\fill (1,3) circle[radius=2pt];   
\fill (1,4) circle[radius=2pt];   

\node at (3,2) {$\alpha' =$};
\fill (4,0) circle[radius=2pt];
\fill (4,1) circle[radius=2pt];
\fill (4,2) circle[radius=2pt];
\fill (4,3) circle[radius=2pt]; 
\fill (4,4) circle[radius=2pt]; 
\fill (5,0) circle[radius=2pt];
\fill (5,1) circle[radius=2pt];
\fill (5,2) circle[radius=2pt];
\fill (5,3) circle[radius=2pt]; 
\fill (5,4) circle[radius=2pt];

\draw (0,4) -- (1,3);
\draw (5,3) -- (4,3);
\draw (5,4) -- (4,4);
\path[-] (1,0) edge [bend left=20] (1,2);
\path[-] (0,2) edge [bend right=20] (0,3);
\path[-] (5,0) edge [bend left=20] (5,2);
\end{tikzpicture}

\caption{$\alpha$ and $\alpha'$}
\end{figure}

Note that $\alpha = \epsilon^k \alpha \alpha'$ where $k$ is a nonnegative integer that depends on the number of right-to-right connections and isolated nodes on the right of $\alpha$.  Since any diagram in $\bY_P$ has a right-to-right connection from node $x$ to node $y$, it is a non-permutation diagram.  Hence, \\
\centerline{$1 \otimes \alpha = 1 \otimes \epsilon^k \alpha \alpha' = 1 \alpha \otimes \epsilon^k \alpha' = 0$}
so $\t \otimes_{\M_n} \bY_{P} = 0$ as needed.\\
To prove the second claim, WLOG, assume $x < y$ and let $X' = X-(X \cap [x,y])$.  Note that $|X'| < |X|$.  \\
To see that $\bY_{P,\{x,y\}}$ is a direct summand of $\dfrac{\M_n}{J_{X'}}$, observe that if $\{a,b\} \in P$ and $\{a,b\} \subseteq X-\{x,y\}$ i.e a right-to-right connection between two nodes labeled by $X-\{x,y\}$, then \\
\centerline{$\bY_P \cap J_{X-\{x,y\}} = \bY_P$}
because no sequence of splices/deletions can remove a right-to-right connection.  Similarly, if $\{a\} \in P$ and $\{a\} \subseteq X-\{x,y\}$ i.e an isolated node labeled by $X-\{x,y\}$, then \\
\centerline{$\bY_P \cap J_{X-\{x,y\}} = \bY_P$} 
because of reason similar to above.  Hence, in both cases, $\bY_{P,\{x,y\}} = 0$ and is trivially a direct summand of $M_n/J_{X'}$.\\
For the last case, assume there is no $\{a,b\}, \{d\} \in P$ such that $\{a,b\}, \{d\} \subseteq X-\{x,y\}$ i.e no right-to-right connections or isolated nodes labeled by $X-\{x,y\}$ inside the connection $x$ to $y$.\\
\indent Let $\gamma \in \bY_P$ be a diagram obtained from the right-link state $Y_P$ by making all defects into left-to-right horizontal connections.  For ex, if $P = \{\{1,3\},\{2\}\}$ and $n = 5$, then 
\begin{figure}[h!]
\begin{tikzpicture}
\node at (-1,2) {$Y_P =$};

\fill (1,0) circle[radius=2pt];
\fill (1,1) circle[radius=2pt];
\fill (1,2) circle[radius=2pt];
\fill (1,3) circle[radius=2pt];   
\fill (1,4) circle[radius=2pt];   

\node at (3,2) {$\gamma =$};
\fill (4,0) circle[radius=2pt];
\fill (4,1) circle[radius=2pt];
\fill (4,2) circle[radius=2pt];
\fill (4,3) circle[radius=2pt]; 
\fill (4,4) circle[radius=2pt]; 
\fill (5,0) circle[radius=2pt];
\fill (5,1) circle[radius=2pt];
\fill (5,2) circle[radius=2pt];
\fill (5,3) circle[radius=2pt]; 
\fill (5,4) circle[radius=2pt];

\draw (0,4) -- (1,4);
\draw (0,3) -- (1,3);
\draw (5,3) -- (4,3);
\draw (5,4) -- (4,4);
\path[-] (1,0) edge [bend left=20] (1,2);
\path[-] (5,0) edge [bend left=20] (5,2);
\end{tikzpicture}

\caption{$Y_P$ and $\gamma$}
\end{figure}
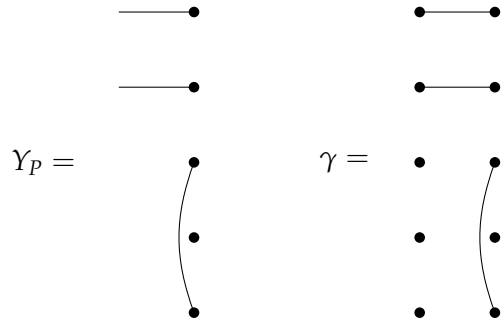

Note that $\gamma \gamma = \epsilon^{k_o}\gamma$ where $k_o $ is a fixed nonnegative integer.  Assume first that $X'$ is nonempty, then any diagram in $J_{X'}$ has a right-to-right connection or isolated node labeled by $X'$ i.e these connections or isolated nodes are outside the connection $\{x,y\}$.  This implies right-multiplying by $\epsilon^{-k_o}\gamma$ takes $J_{X'}$ into $\bY_P \cap J_{X-\{x,y\}}$ and this induces the map \\

\vspace*{-0.1in}
\centerline{$\dfrac{\M_n}{J_{X'}} \xrightarrow[]{\cdot \epsilon^{-k_o} \gamma} \bY_{P,\{x,y\}}$.}

\vspace*{0.2in}
Since there is no $\{a,b\}, \{d\} \in P$ such that $\{a,b\}, \{d\} \subseteq X-\{x,y\}$, any diagram in $\bY_P \cap J_{X-\{x,y\}}$ has to have a right-to-right connection or isolate node labeled by $X'$ and hence, we have an inclusion $\bY_P \cap J_{X-\{x,y\}} \xhookrightarrow{} J_{X'}$.  Along with the inclusion $\bY_P \xhookrightarrow{} \M_n$, these induce the map \\
\centerline{$\bY_{P,\{x,y\}} \xrightarrow[]{} \dfrac{\M_n}{J_{X'}} $.}

\vspace*{0.1in}
The map induced by right-multiplying by $\epsilon^{-k_o} \gamma$ is surjective because for any diagram $\alpha \in \bY_P$, $\alpha \gamma = \epsilon^{k_o}\alpha$.  Since $\gamma \gamma = \epsilon^{k_o}\gamma$, this maps also splits with the splitting map \\

\vspace*{-0.15in}
\centerline{$\bY_{P,\{x,y\}} \xrightarrow[]{} \dfrac{\M_n}{J_{X'}} $}
\vspace*{0.1in}
as above.  Therefore, $\bY_{P,\{x,y\}}$ is a direct summand of $\dfrac{\M_n}{J_{X'}}$.\\
\indent If $X' = \emptyset$, then $\dfrac{\M_n}{J_{X'}} = \M_n$ and we also have $\bY_P \cap J_{X-\{x,y\}} = \emptyset$ because of the assumption that no right-to-right connections or isolated nodes labeled by $X-\{x,y\}$ inside the connection $x$ to $y$ so $\bY_{P,\{x,y\}} = \bY_P$.  The same map as above implies that $\bY_P$ is a direct summand of $\M_n$.
\end{proof}

Similar result holds for $\calA_{X,x}$ and $\calB_{X,x}$.
\begin{lemma} \label{module.Ax.Bx.for.Motzkin}
    Let $\epsilon \in R$ be invertible.  
    \begin{itemize}
        \item
        Let $x\in X\subseteq\{1,\dots,n\}$.
        Then $\t\otimes_{\M_n}\calA_{X,x}=0$ and $\calA_{X,x}$ is a direct summand of $\M_n/J_{X-\{x\}}$.
        \item
        Let $x\in X\subseteq\{1,\dots,n\}$ with $n\geq 2$.
        Then $\t\otimes_{\M_n}\calB_{X,x}=0$.
    \end{itemize}
\end{lemma}
\begin{proof}
    To show that $\t\otimes_{\M_n}\calA_{X,x}=0$, it suffices to show that $\t \otimes_{\M_n} A_x = 0$.  This follows because $\alpha T_x = \epsilon \alpha$ for any diagram $\alpha \in A_x$.  The proof that $\calA_{X,x}$ is a direct summand of $\M_n/J_{X-\{x\}}$ mirrors that of Lemma \ref{modules.lemma} with $\RBr_n$ replaced by $\M_n$.\\
    It suffices to show $\t\otimes_{\M_n} B_{X,x} = 0 $.  Any diagram $\alpha$ in $B_{X,x}$ has $x$ connected to some element $y \in X-\{x\}$ which implies $\alpha \in \bY_P$ for some partition $P$ containing $\{x,y\}$.  The result then follows because $\t\otimes_{\M_n} \bY_P = 0$ by the proof of Lemma \ref{prop.Y_P}.
    
\end{proof}
Analogue of Lemma \ref{tensor.RBr/J} also holds with identical proof where $\RBr_n$ is replaced by $\M_n$. 
\begin{lemma} \label{tensor.M/J}
     Let $J$ be a left ideal of $\M_n$ that is included in $J_n$. Then\\
     \centerline{$\t \otimes_{\M_n} \M_n/J \cong \t.$}
    In particular,\\
    \centerline{$ \Tor^{\M_n}_0(\t,\M_n/J_X) \cong \t $}
    for all $X\subset \{1, \dots, n\}$.
\end{lemma}

We also have a similar resolution of $\dfrac{\M_n}{J_X}$ in terms of $\calA_{X,x}$ and $\calB_{X,x}$.
\begin{proposition} \label{thm.for.inductive.part2}
     Let $X\subseteq\{1,\dots,n\}$, let $x\in X$, and assume $n\geq 2$.
    The following sequence, in which all maps are induced by either an inclusion or an identity map, is a resolution of $\M_n/J_X$.
    \[\xymatrix@R=5pt{
        \dots
        \ar[r]
        &
        0
        \ar[r]
        &
        \calA_{X,x}
        \oplus
        \calB_{X,x}
        \ar[r]
        &
        \M_n/J_{X-\{x\}} 
        \ar[r]
        &
        \M_n/J_X
        \\
        &
        2
        &
        1
        &
        0
        &
        -1
    }\]
    Moreover, applying $\t\otimes_{\M_n}\!\!-$ to the sequence gives a resolution of $\t\otimes_{\M_n} \M_n/J_X$.
\end{proposition}
The proof of this proposition is identical to that of Lemma \ref{thm.for.inductive} with $\RBr_n$ replaced by $\M_n$, Lemmas \ref{module.Ax.Bx.for.Motzkin} and \ref{tensor.M/J} in place of Lemmas \ref{modules.lemma} and \ref{tensor.RBr/J}, respectively.
\begin{lemma} \label{B_X.x.for.Motzkin}
    Let $x \in X \subseteq [n]$ with $n \geq 2$, then the map \\
    \centerline{$\underset{y~\in~X-\{x\}}{\oplus} \big(\oplus~ \bY_{P,\{x,y\}} \big) \to \calB_{X,x}$}
    
    \vspace*{0.1in}
    \noindent induced by inclusion maps $\bY_P \xhookrightarrow{} B_{X,x}$ is an isomorphism of left $\M_n$-modules where the outermost direct sum runs over all $y \in X-\{x\}$; with $y$ fixed, the inner sum runs over all partitions $P$ of $[x,y]$ (or $[y,x]$ if $x < y$) into subsets of size at most 2 such that $\{x,y\} \in P$.
\end{lemma}
\begin{proof}
    Since $y \in X-\{x\}$, $J_{X-\{x,y\}}$ is sent into $J_{X-\{x\}}$.  Since $\{x,y\} \in P$, diagrams in $\bY_P$ have node $y$ connected to node $x$ so these diagrams are also in $B_{X,x}$ and $\bY_P \cap J_{X-\{x,y\}}$ is sent into $B_{X,x} \cap J_{X-\{x\}}$.  The map above is well-defined following from these facts. \\
    \indent To see that the map is surjective, any diagram in $B_{X,x}$ has a right-to-right connection from node $x$ to another node $y \in X-\{x\}$.  Inside this connection, there might be more right-to-right connections which, taken all together, can be identified with a partition $P$ of $[x,y]$ (or $[y,x]$) into subsets of size at most 2 such that $\{x,y\} \in P$.  Hence, the diagram is in the image of the direct summand $\bY_{P,\{x,y\}}$ with the specified $y$ and $P$ above so the map is surjective.  \\
    \indent To see that the map in injective, note that two diagrams from different direct summands $\bY_{P,\{x,y\}}$'s have to be distinct in $\calB_{X,x}$.  This is because for different $y$'s, node $x$ of the two diagrams is connected to two different nodes so the diagram can't be the same.  When $y$'s are the same, different partitions $P$'s of $[x,y]$ (or $[y,x]$) yields different right-to-right connections inside the right-to-right connection of node $x$ and $y$ and hence, two diagram also can't be the same as well.  \\
    \indent Therefore, to show that the map is injective, it suffices to show that for each direct summand $\bY_{P,\{x,y\}}$, the map $\bY_{P,\{x,y\}} \to \calB_{X,x}$ is injective.  This is clear because in order for a diagram $\alpha$ in $\bY_{P,\{x,y\}}$ to be zero in $\calB_{X,x}$, there must be at least a right-to-right connection or an isolated node among nodes of $\alpha$ labeled by $X-\{x\}$.  Since $\alpha$ has node $x$ connected to node $y$, this implies that the right-to-right connection or isolated node must occur among nodes of $\alpha$ labeled by $X-\{x,y\}$.  But this implies $\alpha = 0$ in $\bY_{P,\{x,y\}}$.
\end{proof}

We are now fully equipped to establish the analogue of Theorem \ref{vanishing.RBr} in the context of Motzkin algebras.  
\begin{theorem} \label{vanishing.Motzkin}
    For invertible $\epsilon \in R$ and any $\delta \in R$, let $X \subseteq [n]$, then the groups $\Tor_\ast^{\M_n}(\t,\M_n/J_X)$ vanish in positive degrees.
\end{theorem}
\begin{proof}
   The proof of this is almost identical to that of Theorem \ref{vanishing.RBr} with the following changes: 
   \begin{itemize}
       \item Replace all $\RBr_n$ by $\M_n$.
       \item Replace Proposition \ref{thm.for.inductive} by Proposition \ref{thm.for.inductive.part2}.
       \item $\calA_{X,x}$ is a direct summand of $\M_n/J_{X-\{x\}}$ by Lemma \ref{module.Ax.Bx.for.Motzkin}.
       \item By Lemma \ref{B_X.x.for.Motzkin}, $\calB_{X,\{x\}}$ is a direct sum of $\bY_{P,\{x,y\}}$'s and by Lemma \ref{prop.Y_P}, $\bY_{P,\{x,y\}}$ is a direct summand of $\M_n/J_{X'}$ where $X' \subseteq [n]$ with $|X'| < |X|$.
   \end{itemize}
\end{proof}

\section{Proof of Main Results}
In this section, we present the proofs of our main results, Theorem \ref{main.thm} and Theorem \ref{homology.Motzkin}.  We give the proof of Theorem \ref{homology.Motzkin} first as it is an immediate consequence of Theorem \ref{vanishing.Motzkin}.  \\

\noindent {\bf Proof of Theorem \ref{homology.Motzkin}} \hspace{0.1in}For invertible $\epsilon \in R$ and any $\delta \in R$, we can apply Theorem \ref{vanishing.Motzkin} with $X = [n]$ to yield $\Tor_\ast^{\M_n}(\t,\M_n/J_n) = 0$ for $\ast > 0$.  Since $\M_n/J_n$ is spanned by permutation diagram namely the identity diagram, we see that as a left $\M_n$-module, $\M_n/J_n$ is isomorphic to the trivial module, $\t$.  This implies $\Tor_\ast^{\M_n}(\t,\t) = 0$ for $\ast > 0$ and it's also clear that $\Tor_0^{\M_n}(\t,\t) = \t$. \qed

To prove Theorem \ref{main.thm}, we need the analogue of Shapiro's Lemma for the Rook-Brauer algebras.  
\begin{theorem} \label{strong.shapiro}
    Let $n\geqslant m\geqslant 0$. 
Suppose $\epsilon$ is invertible in $R$ and $\delta \in R$, then the maps
\[
    \iota_\ast\colon
    \Tor_\ast^{R\fS_n}(\t,R\fS_n\otimes_{R\fS_m}\t)
    \longrightarrow
    \Tor_\ast^{\RBr_n}(\t,\RBr_n\otimes_{\RBr_m}\t)
\]
and
\[
    \pi_\ast\colon
    \Tor_\ast^{\RBr_n}(\t,\RBr_n\otimes_{\RBr_m}\t)
    \longrightarrow
    \Tor_\ast^{R\fS_n}(\t,R\fS_n\otimes_{R\fS_m}\t)
\]
are mutually inverse isomorphisms.
\end{theorem}
We can see that Theorem \ref{main.thm} follows immediately from Theorem \ref{strong.shapiro} by choosing $m = n$ and noting that $R\fS_n\otimes_{R\fS_m}\t\cong\t$ and $\RBr_n\otimes_{\RBr_m}\t\cong \t$. \\
For the rest of this section, we will establish preliminary results to prove Theorem \ref{strong.shapiro}.  \\
\indent Recall from Def \ref{def.of.J_X} that $J_m \subseteq \RBr_n$ denotes the left ideal spanned by all diagrams in which among the nodes on the right labeled by $\{n-m+1,\dots, n\}$, there is at least one singleton or a right-to-right connection.  Furthermore, $J_m$ is also a right $\fS_m$-module via the inclusion $\fS_m \subseteq \RBr_m \subseteq \RBr_n$ and this implies $\RBr_n/J_m$ is a right $\fS_m$-module.  
\begin{lemma} \label{lemma.quotient.by.Jm}
    For~$m\leq n$, $\RBr_n/J_m$ is free when regarded as a right $R\fS_m$-module.
\end{lemma}
\begin{proof}
    Any nonzero diagram in $\RBr_n/J_m$ has no singleton or right-to-right connection among nodes in $\{n-m+1,\dots, n\}$ i.e each node in $\{n-m+1,\dots, n\}$ is connected to a distinct node.  These diagrams form a basis of $\RBr_n/J_m$ as a right $R$-module and since $\fS_m$ acts freely on these diagram, $\RBr_n/J_m$ is a free $R\fS_m$-module.
\end{proof}

\begin{lemma} \label{lemma.quotient.by.Jm.tensor}
For~$m\leq n$, there is an isomorphism of left $\RBr_n$-modules
    \[
    \RBr_n/J_m\otimes_{R\fS_m}\t \cong \RBr_n\otimes_{\RBr_m}\t 
    \]
    where $(b+J_m)\otimes r\in\RBr_n/J_m\otimes_{R\fS_m}\t$ is mapped to $b\otimes r\in \RBr_n\otimes_{\RBr_m}\t$.
    
\end{lemma}
The proof of this Lemma is almost identical to that of Lemma 5.3 of \cite{bhp} with $P_n$ replaced by $\RBr_n$ and $P_m$ replaced by $\RBr_m$.\\ \\
{\bf Proof of Theorem \ref{strong.shapiro}:}
Under the assumption of Theorem \ref{strong.shapiro}, we have 
\[
\Tor^{\RBr_n}_\ast (\t, \RBr_n/J_m)=\begin{cases} \t &\mbox{if } \ast= 0 \\
0 & \mbox{if } \ast>0 \end{cases}
\]
by Theorem \ref{vanishing.RBr}.  The proof now follows exactly as that of Theorem 4.1 of \cite{bhp1} with $\RBr_n$ replacing $\Br_n$ and using Lemma \ref{lemma.quotient.by.Jm.tensor}, Theorem \ref{vanishing.RBr} in place of Lemma 4.4 and Theorem 3.2, respectively. \qed


\newpage 
\begin{bibdiv}
\begin{biblist}
\bib{bhp}{article}{
	author={Boyd, Rachel Jane},
	author={Hepworth,Richard},
	author={Patzt, Peter},
	title={The homology of the partition algebras},
	journal={Pacific Journal of Mathematics},
	volume={327},
	date={2023},
	number={1},
	pages={1-27},
	eprint={https://msp.org/pjm/2023/327-1/p01.xhtml} }

\bib{bh}{article}{
	author={Boyd, Rachel Jane},
	author={Hepworth,Richard},
	title={The homology of the Temperley-Lieb algebras},
	journal={Geometry and Topology},
	volume={28},
	date={2024},
	number={3},
	pages={1437-1499},
	eprint={https://msp.org/gt/2024/28-3/gt-v28-n3-p11-s.pdf} }

    
\bib{bhp1}{article}{
	author={Boyd, Rachel Jane},
	author={Hepworth,Richard},
    author = {Patzt, Peter},
	title={The homology of the Brauer algebras},
	journal={Sel. Math. New Ser.},
	volume={27},
	date={2021},
	number={85},
	eprint={https://link.springer.com/article/10.1007/s00029-021-00697-4} }

\bib{b}{article}{
	author={Boyde, Guy},
	title={Idempotents and homology of diagram algebras},
	journal={Math. Ann.},
	volume={391},
	date={2025},
	pages={2173–2207 },
	eprint={https://link.springer.com/article/10.1007/s00208-024-02960-3} }



\bib{Patzt}{article}{
   author={Patzt, Peter} 
   title={Representation stability for diagram algebras},
   journal = {Journal of Algebra, vol 638},
   pages = {625-669},
   year={2024},
   ISSN = {0021-8693},
   eprint={{https://doi.org/10.1016/j.jalgebra.2023.09.017}},}

\bib{fg}{article}{
   author={Fisher, Andrew},
   author = {Graves, Daniel}
   title={Cohomology of Tanabe algebras},
   year={2024},
   eprint={{https://arxiv.org/abs/2410.00599}} 
}

\bib{fg1}{article}{
   author={Fisher, Andrew},
   author = {Graves, Daniel},
   title={Cohomology of diagram algebras},
   year={2024},
   eprint={{https://arxiv.org/abs/2412.14887}} 
}

\bib{nak}{article}{
   author={Nakaoka, Minoru},
   title={Decomposition Theorem for Homology Groups of Symmetric Groups},
   journal = {Annals of Mathematics, vol 71},
   pages = {16-42},
   year={1960},
   eprint={{https://www.jstor.org/stable/1969878}} 
}

\bib{mt}{article}{
   author={Muljat, Anthony},
   author = {Ta, Khoa},
   title={Noetherianity of Diagram Algebras},
   year={2024},
   eprint={{https://doi.org/10.48550/arXiv.2409.10885}}  }
  
\end{biblist}
\end{bibdiv}
\end{document}